\documentclass[graybox]{svmult}

\smartqed
\usepackage{mathptmx}       
\usepackage{helvet}         
\usepackage{courier}        
\usepackage{type1cm}        
\usepackage{graphicx}       

\usepackage{array,colortbl}
\usepackage{amsmath,amsfonts,amssymb,bm} 
\DeclareFontFamily{U}{mathx}{\hyphenchar\font45}
\DeclareFontShape{U}{mathx}{m}{n}{
      <5> <6> <7> <8> <9> <10>
      <10.95> <12> <14.4> <17.28> <20.74> <24.88>
      mathx10
      }{}
\DeclareSymbolFont{mathx}{U}{mathx}{m}{n}
\DeclareFontSubstitution{U}{mathx}{m}{n}
\DeclareMathAccent{\widecheck}      {0}{mathx}{"71}

\newcommand{\R}{\mathbb{R}} 
\newcommand{\N}{\mathbb{N}} 
\newcommand{\F}{\mathbb{F}} 
\newcommand{\bszero}{\boldsymbol{0}} 
\newcommand{\bst}{\boldsymbol{t}}    
\newcommand{\bsx}{\boldsymbol{x}}    
\newcommand{\bsz}{\boldsymbol{z}}    
\newcommand{\bsDelta}{\boldsymbol{\Delta}}    

\DeclareSymbolFont{bbold}{U}{bbold}{m}{n}
\DeclareSymbolFontAlphabet{\mathbbold}{bbold}
\newcommand{\ind}{\mathbbold{1}}
\newcommand{\mC}{\mathsf{C}}

\usepackage{microtype} 

\usepackage[colorlinks=true,linkcolor=black,citecolor=black,urlcolor=black]{hyperref}
\usepackage{bookmark}
\makeatletter 
  \providecommand*{\toclevel@author}{999}
  \providecommand*{\toclevel@title}{0}
\makeatother

\usepackage{mathtools}
\DeclareMathOperator{\ok}{ok}

\begin{document}
\spnewtheorem{algo}{Algorithm}{\bf}{\rm}
\newcommand{\bsi}{\boldsymbol{i}}    
\newcommand{\bsk}{\boldsymbol{k}}    
\newcommand{\bsl}{\boldsymbol{l}}    
\newcommand{\bsr}{\boldsymbol{r}}    
\newcommand{\bsnu}{\boldsymbol{\nu}}    
\newcommand{\cc}{\mathcal{C}}
\newcommand{\cl}{\mathcal{L}}
\newcommand{\cn}{\mathcal{N}}
\newcommand{\Order}{\mathcal{O}}
\newcommand{\cp}{\mathcal{P}}
\newcommand{\cx}{\mathcal{X}}
\newcommand{\natm}{\N_{0,m}}
\newcommand{\cube}{[0,1)^d}
\newcommand{\hf}{\hat{f}}
\newcommand{\rf}{\mathring{f}}
\newcommand{\tf}{\tilde{f}}
\newcommand{\hg}{\hat{g}}
\newcommand{\hI}{\hat{I}}
\newcommand{\tvk}{\tilde{\bsk}}
\newcommand{\hS}{\widehat{S}}
\newcommand{\tS}{\widetilde{S}}
\newcommand{\wcS}{\widecheck{S}}
\newcommand{\rnu}{\mathring{\nu}}
\newcommand{\tnu}{\widetilde{\nu}}
\newcommand{\hnu}{\widehat{\nu}}
\newcommand{\homega}{\widehat{\omega}}
\newcommand{\wcomega}{\mathring{\omega}}
\newcommand{\fC}{\mathfrak{C}}
\newcommand{\nodes}{\{\bsz_i\}_{i=0}^{\infty}}
\newcommand{\nodesn}{\{\bsz_i\}_{i=0}^{n-1}}
\newcommand{\norm}[1]{\ensuremath{\left \lVert #1 \right \rVert}} \newcommand{\bigabs}[1]{\ensuremath{\bigl \lvert #1 \bigr \rvert}}
\newcommand{\Bigabs}[1]{\ensuremath{\Bigl \lvert #1 \Bigr \rvert}}
\newcommand{\biggabs}[1]{\ensuremath{\biggl \lvert #1 \biggr \rvert}}
\newcommand{\Biggabs}[1]{\ensuremath{\Biggl \lvert #1 \Biggr \rvert}}
\newcommand{\ip}[3][{}]{\ensuremath{\left \langle #2, #3 \right \rangle_{#1}}}

\allowdisplaybreaks

\title*{Reliable Adaptive Cubature Using Digital Sequences}
\author{Fred J. Hickernell \and Llu\'is Antoni Jim\'enez Rugama}
\institute{Fred J. Hickernell \and Llu\'is Antoni Jim\'enez Rugama
\at Department of Applied Mathematics,  Illinois Institute of Technology, 10 W. 32$^{\text{nd}}$ Street, E1-208, Chicago, IL 60616, USA 
\email{hickernell@iit.edu}, \email{ljimene1@hawk.iit.edu}}
\maketitle

\emph{In honor of Ilya M. Sobol'}

\bigskip

\abstract{Quasi-Monte Carlo cubature methods often sample the integrand using Sobol' (or other digital) sequences to obtain higher accuracy than IID sampling.     An important question is how to conservatively estimate the error of a digital sequence cubature so that the sampling can be terminated when the desired tolerance is reached.  We propose an error bound based on the discrete Walsh coefficients of the integrand and use this error bound to construct an adaptive digital sequence cubature algorithm.  The error bound and the corresponding algorithm are guaranteed to work for integrands lying in a cone defined in terms of their true Walsh coefficients.  Intuitively, the inequalities defining the cone imply that the ordered Walsh coefficients do not dip down for a long stretch and then jump back up.  An upper bound on the cost of our new algorithm is given in terms of the \emph{unknown} decay rate of the Walsh coefficients.} 

\section{Introduction}

Quasi-Monte Carlo cubature rules approximate multidimensional integrals over the unit cube by an equally weighted sample average of the integrand values at the first $n$ nodes from some sequence $\nodes$.  This node sequence should be chosen to minimize the error, and for this one can appeal to Koksma-Hlawka type error bounds of the form 
\begin{equation} \label{KHineq}
\biggabs{\int_{\cube} f(\bsx) \D \bsx - \frac{1}{n} \sum_{i=0}^{n-1} f(\bsz_i)} \le D(\nodesn) V(f).
\end{equation}
The discrepancy, $D(\nodesn)$, measures how far the empirical distribution of the first $n$ nodes differs from the uniform distribution.  The variation, $V(f)$, is some semi-norm of the integrand, $f$.  The definitions of the discrepancy and variation are linked to each other. Examples of such error bounds are given by \cite[Chap.\ 2--3]{DicPil10a}, \cite{Hic97a}, \cite[Sec.\ 5.6]{Lem09a}, \cite[Chap.\ 2--3]{Nie92}, and \cite[Chap.\ 9]{NovWoz10a}.

A practical problem is how large to choose $n$ so that the absolute error is smaller than some user-defined tolerance, $\varepsilon$.  Error bounds of the form \eqref{KHineq} do not help in this regard because it is too hard to compute $V(f)$, which is typically defined in terms of integrals of mixed partial derivatives of $f$.

This article addresses the challenge of reliable error estimation for quasi-Monte Carlo cubature based on digital sequences, of which Sobol' sequences are the most popular example.  The vector space structure underlying these digital sequences facilitates a convenient expression for the error in terms of the (Fourier)-Walsh coefficients of the integrand.  Discrete Walsh coefficients can be computed efficiently, and their decay provides a reliable cubature error estimate.  Underpinning this analysis is the assumption that the integrands lie in a cone  defined in terms of their true Walsh coefficients; see \eqref{conecond}.

The next section introduces digital sequences and their underlying algebraic structure.  Sect.\ \ref{WaveWalshsec} explains how the cubature error using digital sequences as nodes can be elegantly formulated in terms of the Walsh series representation of the integrand. Our contributions begin in Sect.\ \ref{ErrEstsec}, where we derive a reliable data-based cubature error bound for a cone of integrands, \eqref{errbd}, and an adaptive cubature algorithm based on that error bound, Algorithm \ref{adapalgo}.  The cost of the algorithm is also represented in terms of the unknown decay of the Walsh series coefficients and the error tolerance in Theorem \ref{adapalgothm}.  A numerical example and discussion then conclude this article. A parallel development for cubature based on lattice rules is given in \cite{JimHic16a}.

\section{Digital Sequences}

The integrands considered here are defined over the half open $d$-dimensional unit cube, $\cube$.  For integration problems on other domains one may often transform the integration variable so that the problem is defined on $\cube$.  See \cite{Caf98,HicSloWas03c,HicSloWas03b,HicSloWas03a,HicSloWas03e} for some discussion of variable transformations and the related error analysis.  The example in Sect.\ \ref{numexpsec} also employs a variable transformation.

Digital sequences are defined in terms of digitwise addition.  Let $b$ be a prime number; $b=2$ is the choice made for Sobol' sequences.  Digitwise addition, $\oplus$, and negation, $\ominus,$ are defined in terms of the proper $b$-ary expansions of points in $\cube$:
\begin{gather*}
\bsx = \left(\sum_{\ell=1}^{\infty} x_{j\ell} b^{-\ell}\right)_{j=1}^d, \quad \bst = \left(\sum_{\ell=1}^{\infty} t_{j\ell} b^{-\ell}\right)_{j=1}^d, \qquad x_{j\ell}, t_{j\ell} \in \F_b:=\{0, \ldots, b-1\},\\
\bsx \oplus \bst = \left(\sum_{\ell=1}^{\infty} [(x_{j\ell} + t_{j\ell}) \bmod b] b^{-\ell} \pmod{1} \right)_{j=1}^d,\qquad \bsx \ominus \bst:=\bsx \oplus (\ominus \bst),\\
\ominus \bsx = \left(\sum_{\ell=1}^{\infty} [-x_{j\ell} \bmod b] b^{-\ell}\right)_{j=1}^d,  \qquad a \bsx:=\underbrace{\bsx \oplus \cdots \oplus \bsx}_{a \text{ times}}\ \forall a \in \F_b.
\end{gather*}

We do not have associativity for all of $\cube$.  For example, for $b=2$, 
\begin{gather*}
1/6={}_20.001010\ldots, \quad  1/3={}_20.010101\ldots,\quad 1/2 = {}_20.1000\dots\\
1/3\oplus 1/3={}_20.00000\ldots =0, \quad 1/3\oplus1/6={}_20.011111\ldots=1/2,\\
 (1/3\oplus1/3)\oplus1/6=0 \oplus 1/6 = 1/6, \quad 1/3\oplus(1/3\oplus1/6)=1/3 \oplus 1/2 = 5/6.
\end{gather*}
This lack of associativity comes from the possibility of digitwise addition resulting in an infinite trail of digits $b-1$, e.g., $1/3\oplus1/6$ above.  

Define the Boolean operator that checks whether digitwise addition of two points does not result in an infinite trail of digits $b-1$:
\begin{equation} \label{okdef}
\ok(\bsx,\bst)=\begin{cases}
\text{true}, & \min_{j=1,\ldots, d} \sup\{ \ell : [(x_{j\ell} + t_{j\ell}) \bmod b ] \neq b-1\}=\infty, \\
\text{false}, & \text{otherwise}.
\end{cases}
\end{equation}
If $\cp \subset \cube$ is some set that is closed under $\oplus$ and $\ok(\bsx,\bst) = $ true for all $\bsx,\bst \in \cp$, then associativity holds for all points in $\cp$.  Moreover, $\cp$ is an Abelian group and also a vector space over the field $\F_b$.

Suppose that $\cp_\infty =\{\bsz_i\}_{i=0}^{\infty} \subset \cube$ is such a vector space that satisfies the following additional conditions:
\begin{subequations} \label{cpinfvector}
\begin{gather}
\{\bsz_1, \bsz_b, \bsz_{b^2}, \ldots\} \text{ is a set of linearly independent points}, \\
\bsz_i = \sum_{\ell=0}^{\infty} i_\ell \bsz_{b^\ell}, \qquad \text{where }i = \sum_{\ell=0}^{\infty} i_{\ell} b^{\ell} \in \N_0, \quad i_\ell \in \F_b.
\end{gather}
\end{subequations}
Such a $\cp_\infty$ is called a \emph{digital sequence}.  Moreover, any $\cp_m := \{\bsz_i\}_{i=0}^{b^m-1}$ is a subspace of $\cp_\infty$ and is called a \emph{digital net}.  From this definition it is clear that
\[
\cp_0=\{\bszero\}\subset \cp_1=\{\bszero,\bsz_1, \ldots, (b-1) \bsz_1\}  \subset \cp_2 \subset \cdots \subset \cp_{\infty} =\{\bsz_i\}_{i=0}^{\infty}.
\]

This \emph{digital sequence} definition is equivalent to the traditional one in terms of generating matrices. By \eqref{cpinfvector} and according to the $b$-ary expansion notation introduced earlier, the $m,\ell$ element of generating matrix, $\mC_j$, for the $j^{\text{th}}$ coordinate is the $\ell^{\text{th}}$ binary digit of the $j^{\text{th}}$ element of $\bsz_{b^{m-1}}$, i.e., 
\[
\mC_j=
\begin{pmatrix}
(z_{1})_{j1} & (z_{b})_{j1} & (z_{b^2})_{j1} & \cdots \\
(z_{1})_{j2} & (z_{b})_{j2} & (z_{b^2})_{j2} & \cdots \\
(z_{1})_{j3} & (z_{b})_{j3} & (z_{b^2})_{j3} & \cdots \\
\vdots & \vdots & \vdots & \ddots
\end{pmatrix},
\qquad \text{for }j = 1,\dots,d.
\]

The Sobol' sequence works in base $b=2$ and makes a careful choice of the basis $\{\bsz_1, \bsz_2, \bsz_{4}, \ldots\}$ so that the points are evenly distributed. Figure \ref{Sobolfig}a) displays the initial points of the two-dimensional Sobol' sequence.  In Figure \ref{Sobolfig}b) the Sobol' sequence has been linearly scrambled to obtain another digital sequence and then digitally shifted.

\begin{figure}
\centering
\begin{tabular}{>{\centering}p{5cm}>{\centering}p{5cm}}
\includegraphics[width=5cm]{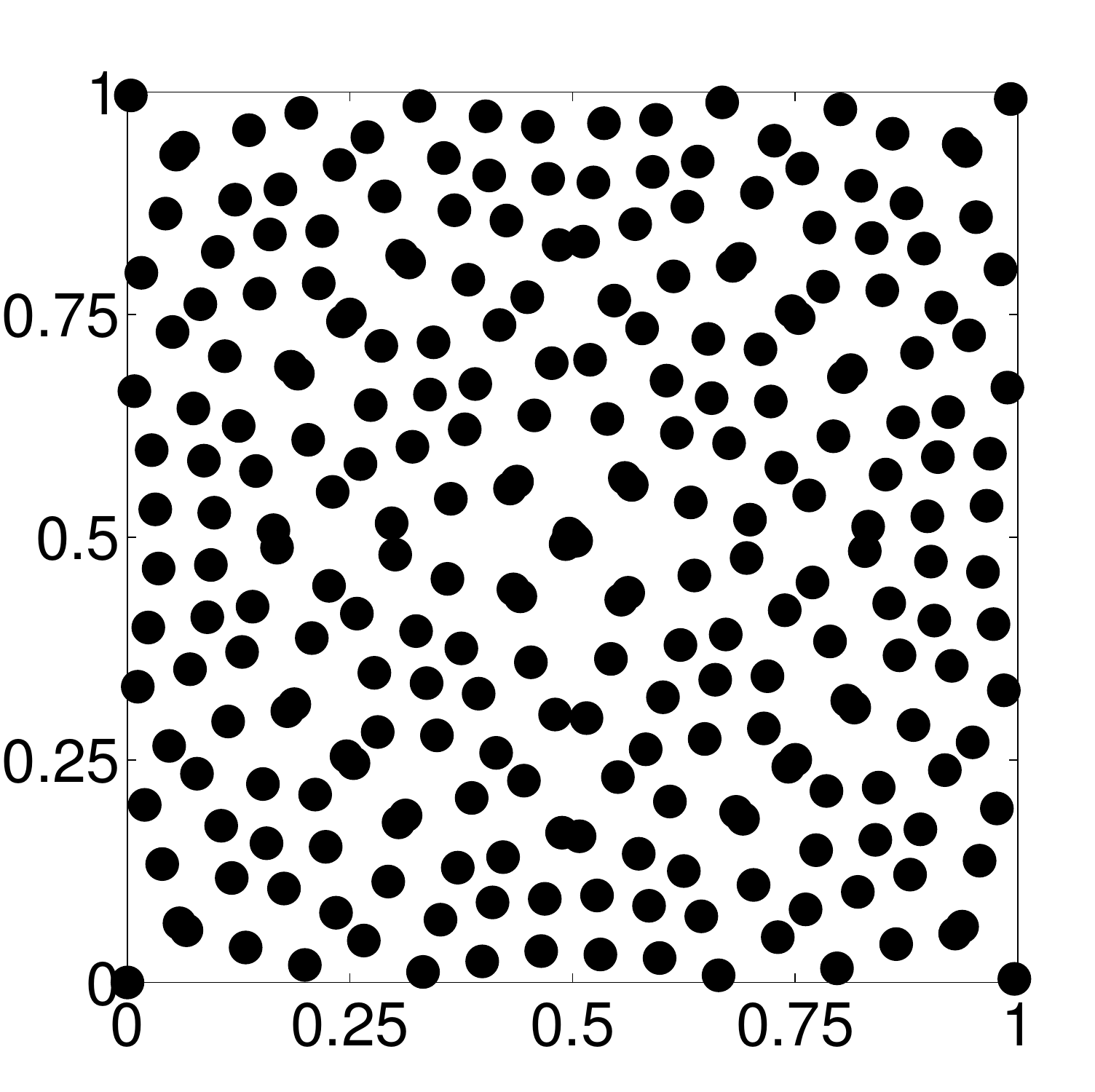} &
\includegraphics[width=5cm]{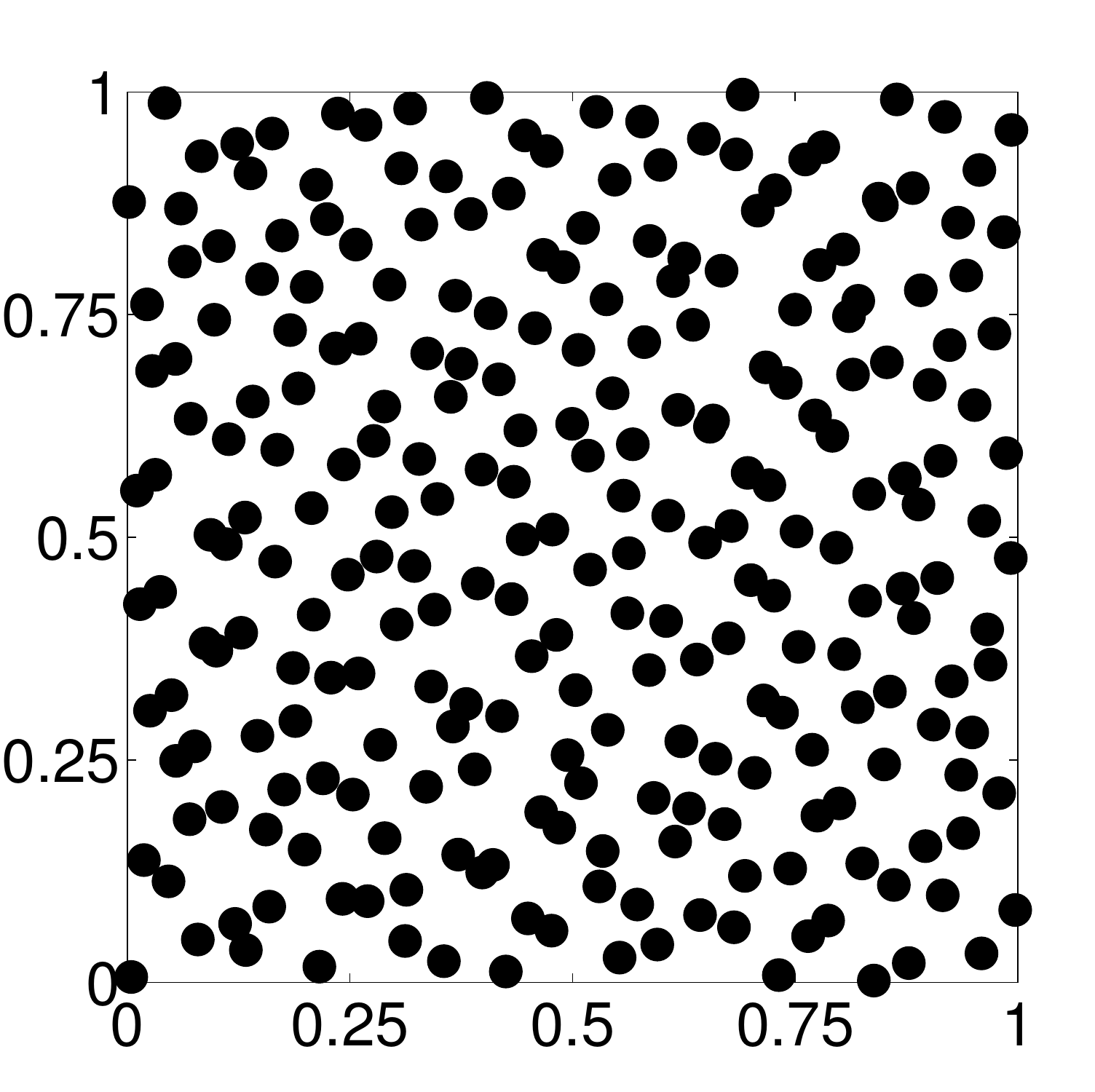}\tabularnewline
a) & b)
\end{tabular}
\caption{a) 256 Sobol' points, b) 256 scrambled and digitally shifted Sobol' points. \label{Sobolfig}}
\end{figure}

\section{Walsh Series} \label{WaveWalshsec}
Non-negative integer vectors are used to index the Walsh series for the integrands.  The set $\N_0^d$ is a vector space under digitwise addition, $\oplus$, and the field $\F_b$.  Digitwise addition and negation are defined as follows for  all $\bsk, \bsl \in \N_0^d$: 
\begin{gather*}
\bsk =  \left(\sum_{\ell=0}^{\infty} k_{j\ell} b^{\ell}\right)_{j=1}^d, \quad \bsl =  \left(\sum_{\ell=0}^{\infty} l_{j\ell} b^{\ell}\right)_{j=1}^d, \qquad k_{j\ell}, l_{j\ell} \in \F_b, \\
\bsk \oplus \bsl = \left(\sum_{\ell=0}^{\infty} [(k_{j\ell} + l_{j\ell}) \bmod b] b^{\ell}\right)_{j=1}^d, \\
\ominus \bsk = \left(\sum_{\ell=0}^{\infty} (b-k_{j\ell}) b^{\ell}\right)_{j=1}^d, \qquad
a \bsk:=\underbrace{\bsk \oplus \cdots \oplus \bsk}_{a \text{ times}}\ \forall a \in \F_b.
\end{gather*}

For each \emph{wavenumber} $\bsk \in \N_0^d$ a function $\ip{\bsk}{\cdot}: \cube \to \F_b$ is defined as
\begin{subequations} \label{bilinear}
\begin{equation}
\ip{\bsk}{\bsx} := \sum_{j=1}^{d} \sum_{\ell=0}^{\infty} k_{j\ell}x_{j,\ell+1}  \pmod b.
\end{equation}
For all points $\bst, \bsx \in \cube$, wavenumbers $\bsk, \bsl \in \N_0^d$, and $a \in \F_b$, it follows that
\begin{gather}
\ip{\bsk}{\bszero} = \ip{\bszero}{\bsx} = 0,\\
\ip{\bsk}{a \bsx \oplus \bst} = a\ip{\bsk}{\bsx} + \ip{\bsk}{\bst} \pmod b \quad \text{if } \ok(a\bsx,\bst) \label{bilinearlinxprop} \\
\ip{a \bsk \oplus \bsl}{\bsx} = a\ip{\bsk}{\bsx} + \ip{\bsl}{\bsx} \pmod b, \label{bilinearlinkprop}\\
\ip{\bsk}{\bsx} = 0 \ \forall \bsk \in \N_0^d \ \implies \ \bsx=\bszero.
\end{gather}
\end{subequations}
The digital sequences $\cp_{\infty}=\{\bsz_i\}_{i=0}^{\infty}$ considered here are assumed to contain sufficient points so that 
\begin{equation}
\ip{\bsk}{\bsz_i} =  0 \ \forall i \in \N_0   \ \implies \ \bsk=\bszero. \label{netpropc}
\end{equation}

Defining $\natm:=\{0, \ldots, b^{m}-1\}$, the \emph{dual net} corresponding to the net $\cp_m$ is the set of all wavenumbers for which $\ip{\bsk}{\cdot}$ maps the whole net to $0$:  
\begin{align*}
\cp^{\perp}_m &:= \{\bsk \in \N_0^d : \ip{\bsk}{\bsz_i} = 0, \ i\in \natm\} \\
&= \{\bsk \in \N_0^d : \ip{\bsk}{\bsz_{b^{\ell}}} = 0, \ \ell=0, \ldots, m-1\}.
\end{align*}
The properties of the bilinear transform defined in \eqref{bilinear} imply that the dual nets $\cp^{\perp}_m$ are subspaces of each other: 
\[
\cp_0^{\perp}=\N_0^d \supset \cp_1^{\perp}  \supset \cdots \supset \cp_{\infty}^{\perp} =\{\bszero\}.
\]


The integrands are assumed to belong to some subset of $L^2(\cube)$, the space of square integrable functions.  The $L^2$ inner product is defined as 
\[
\ip[2]{f}{g} = \int_{\cube} f(\bsx) \overline{g(\bsx)} \, \D \bsx.
\]
The Walsh functions $\{\exp(2 \pi \sqrt{-1} \ip{\bsk}{\cdot}/b) : \bsk \in \N_0^d\}$ \cite[Appendix A]{DicPil10a} are a complete orthonormal \emph{basis} for $L^2(\cube)$.  Thus, any function in $L^2$ may be written in series form as
\begin{equation} \label{Fourierdef}
f(\bsx) = \sum_{\bsk \in \N_0^d} \hf(\bsk) \E^{2 \pi \sqrt{-1} \ip{\bsk}{\bsx}/b}, \quad \text{where } \hf(\bsk) := \ip[2]{f}{\E^{2 \pi \sqrt{-1} \ip{\bsk}{\cdot}/b}},
\end{equation}
and the $L^2$ inner product of two functions  is the $\ell^2$ inner product of their Walsh series coefficients:
\[
\ip[2]{f}{g} = \sum_{\bsk \in \N_0^d} \hf(\bsk)\overline{\hg(\bsk)} =: \ip[2]{\bigl(\hf(\bsk)\bigr)_{\bsk \in \N_0^d}}{\bigl ( \hg(\bsk)\bigr )_{\bsk \in \N_0^d}}.
\]

Since the digital net $\cp_m$ is a group under $\oplus$, one may derive a useful formula for the average of a Walsh function sampled over a net. For all wavenumbers $\bsk \in \N_0^d$ and all $\bsx \in \cp_m$ one has
\begin{align*} 
\nonumber
0 & = \frac{1}{b^m} \sum_{i=0}^{b^m-1} [\E^{2 \pi \sqrt{-1} \ip{\bsk}{\bsz_i}/b} - \E^{2 \pi \sqrt{-1} \ip{\bsk}{\bsz_i \oplus \bsx}/b}]\\
\nonumber
& = \frac{1}{b^m} \sum_{i=0}^{b^m-1} [\E^{2 \pi \sqrt{-1} \ip{\bsk}{\bsz_i}/b} - \E^{2 \pi \sqrt{-1} \{\ip{\bsk}{\bsz_i}+\ip{\bsk}{\bsx}\}/b}] \quad \text{by } \eqref{bilinearlinxprop}\\
\label{sumeq}
& = [1 - \E^{2 \pi \sqrt{-1} \ip{\bsk}{\bsx}/b}] \frac{1}{b^m} \sum_{i=0}^{b^m-1}  \E^{2 \pi \sqrt{-1} \ip{\bsk}{\bsz_i}/b}.
\end{align*}
By this equality it follows that the average of the sampled Walsh function values is either one or zero, depending on whether the wavenumber is in the dual net or not:
\begin{equation}
\frac{1}{b^m} \sum_{i=0}^{b^m-1}  \E^{2 \pi \sqrt{-1} \ip{\bsk}{\bsz_i}/b} = \ind_{\cp_m^{\perp}}(\bsk) = \begin{cases} 1 , & \bsk \in \cp_m^{\perp}\\
 0,  & \bsk \in \N_0^d \setminus \cp_m^{\perp}.
 \end{cases} \label{Walsherr}
\end{equation}

Multivariate integrals may be approximated by the average  of the integrand sampled over a digitally shifted digital net, namely,
\begin{equation} \label{cubaturedef}
\hI_m(f) := \frac{1}{b^m} \sum_{i=0}^{b^m-1} f(\bsz_i \oplus\bsDelta).
\end{equation}
Under the assumption that $\ok(\bsz_i,\bsDelta) = \text{true}$ (see \eqref{okdef}) for all $i \in \N_0$, it follows that the error of this cubature rule is the sum of the Walsh coefficients of the integrand over those wavenumbers in the dual net:
\begin{align}
\nonumber
\biggabs{ \int_{\cube} f(\bsx) \, \D \bsx - \hI_m(f)} 
& = \Biggabs {\hf(\bszero) - \sum_{\bsk \in \N_0^d} \hf(\bsk) \hI_m\left(\E^{2 \pi \sqrt{-1} \ip{\bsk}{\cdot}/b}\right)} \\
\nonumber
& = \Biggabs {\hf(\bszero) - \sum_{\bsk \in \N_0^d} \hf(\bsk) \ind_{\cp_m^{\perp}}(\bsk) \E^{2 \pi \sqrt{-1} \ip{\bsk}{\bsDelta}/b}} \\ 
& = \Biggabs {\sum_{\bsk \in \cp_m^{\perp}\setminus \{\bszero\} } \hf(\bsk) \E^{2 \pi \sqrt{-1} \ip{\bsk}{\bsDelta}/b}}. \label{err1}
\end{align}
Adaptive Algorithm \ref{adapalgo} that we construct in Sect.\ \ref{ErrEstsec} works with this expression for the cubature error in terms of Walsh coefficients.  


Although the true Walsh series coefficients are generally not known, they can be estimated by the discrete Walsh transform, defined as follows:
\begin{align}
\nonumber
\tf_m(\bsk) 
&:= \hI_m\left(\E^{-2 \pi \sqrt{-1} \ip{\bsk}{\cdot}/b} f(\cdot) \right) = \frac{1}{b^m} \sum_{i=0}^{b^m-1} \E^{-2 \pi \sqrt{-1} \ip{\bsk}{\bsz_i\oplus \bsDelta}/b} f(\bsz_i\oplus \bsDelta) \\
\nonumber
&= \frac{1}{b^m}  \sum_{i=0}^{b^m-1} \left[\E^{-2 \pi \sqrt{-1} \ip{\bsk}{\bsz_i\oplus \bsDelta}/b}\sum_{\bsl \in \N_0^d} \hf(\bsl) \E^{2 \pi \sqrt{-1} \ip{\bsl}{\bsz_i\oplus \bsDelta}/b} \right] \\
\nonumber
& = \sum_{\bsl \in \N_0^d} \hf(\bsl)  \frac{1}{b^m}  \sum_{i=0}^{b^m-1}  \E^{2 \pi \sqrt{-1} \ip{\bsl \ominus \bsk}{\bsz_i\oplus \bsDelta}/b} \\
\nonumber 
& = \sum_{\bsl \in \N_0^d} \hf(\bsl) \E^{2 \pi \sqrt{-1} \ip{\bsl \ominus \bsk}{\bsDelta}/b}  \frac{1}{b^m}  \sum_{i=0}^{b^m-1}  \E^{2 \pi \sqrt{-1} \ip{\bsl \ominus \bsk}{\bsz_i}/b} \\
\nonumber
& = \sum_{\bsl \in \N_0^d} \hf(\bsl) \E^{2 \pi \sqrt{-1} \ip{\bsl \ominus \bsk}{\bsDelta}/b} \ind_{\cp_m^{\perp}}(\bsl \ominus \bsk) \\
\nonumber
& = \sum_{\bsl \in \cp^{\perp}_m} \hf(\bsk\oplus\bsl) \E^{2 \pi \sqrt{-1} \ip{\bsl}{\bsDelta}/b} \\
&= \hf(\bsk) + \sum_{\bsl \in \cp^{\perp}_m\setminus \{\bszero\}} \hf(\bsk\oplus\bsl) \E^{2 \pi \sqrt{-1} \ip{\bsl}{\bsDelta}/b}, \qquad \forall \bsk \in \N_0^d. \label{tfassum}
\end{align}
The discrete transform, $\tf_m(\bsk)$ is equal to the true Walsh transform, $\hf(\bsk)$, plus \emph{aliasing} terms proportional to $\hf(\bsk\oplus\bsl)$ where $\bsl$ is a nonzero wavenumber in the dual net.

\section{Error Estimation and an Adaptive Cubature Algorithm} \label{ErrEstsec}

\subsection{Wavenumber Map} \label{wavenummapsec}

Since the discrete Walsh transform has aliasing errors, some assumptions must be made about how quickly the true Walsh coefficients decay and which coefficients are more important.  This is done by way of a map of the non-negative integers \emph{onto} the space of all wavenumbers, $\tvk: \N_0 \to \N_0^d$, according to the following algorithm. 
\begin{algo} \label{wavenummapalgo} Given a digital sequence, $\cp_{\infty} = \{\bsz_i\}_{i=0}^{\infty}$ define $\tvk: \N_0 \to \N_0^d$ as follows:
\begin{description}
\item[\textbf{Step 1.}] Define $\tvk(0)=\bszero$.

\item[\textbf{Step 2.}] For $m=0, 1, \ldots$ \\
\hspace*{1.3cm} For $\kappa = 0, \ldots, b^{m} -1 $ \\
\hspace*{1.6cm} Choose the values of $\tvk(\kappa+b^{m}), \ldots, \tvk(\kappa+(b-1)b^{m})$ from the sets
\[
\left\{\bsk \in \N_0^d : \bsk \ominus \tvk(\kappa) \in \cp_{m}^{\perp}, \ip{\bsk \ominus \tvk(\kappa)}{\bsz_{b^m}}=a \right\}, \quad a=1, \ldots, b-1,
\]
\hspace*{1.6cm} but not necessarily in that order.
\end{description}
\end{algo}

There is some flexibility in the choice of this map.  One might choose $\tvk$ to map smaller values of $\kappa$ to smaller values of $\bsk$ based on some standard measure of size such as that given in \cite[(5.9)]{DicPil10a}. The motivation is that larger $\kappa$ should generally lead to smaller $\hf(\tvk(\kappa))$.  We use Algorithm \ref{pointeralgo} below to construct this map implicitly.

To illustrate the initial steps of Algorithm \ref{wavenummapalgo}, consider the Sobol' points in dimension 2. In this case, $\bsz_1=(1/2,1/2)$, $\bsz_2=(1/4,3/4)$ and $\bsz_4=(1/8,5/8)$. For $m=\kappa=0$, one needs 
\[
\tvk(1)\in\left\{\bsk \in \N_0^d : \bsk \ominus \tvk(0) \in \cp_{0}^{\perp}, \ip{\bsk \ominus \tvk(0)}{\bsz_{1}}=1 \right\}=\left\{\bsk \in \N_0^d : \ip{\bsk}{\bsz_{1}}=1 \right\}. 
\]
Thus, one may choose $\tvk(1)=(1,0)$. Next, $m=1$ and $\kappa=0$  leads to
\begin{multline*}
\tvk(2)\in\left\{\bsk \in \N_0^d : \bsk \ominus \tvk(0) \in \cp_{1}^{\perp}, \ip{\bsk \ominus \tvk(0)}{\bsz_{2}}=1 \right\} \\
=\left\{\bsk \in \N_0^d : \bsk \in \cp_{1}^{\perp}, \ip{\bsk}{\bsz_{2}}=1 \right\}. 
\end{multline*}
Hence, we can take $\tvk(2):=(1,1)$. Continuing with $m=\kappa=1$ requires \[
\tvk(3)\in\left\{\bsk \in \N_0^d : \bsk \ominus \tvk(1) \in \cp_{1}^{\perp}, \ip{\bsk \ominus \tvk(1)}{\bsz_{2}}=1 \right\},
\]
so the next choice can be $\tvk(3):=(0,1)$.

Introducing the shorthand notation $\hf_{\kappa} :=\hf(\tvk(\kappa))$  and $\tf_{m,\kappa} := \tf_m(\tvk(\kappa))$, the aliasing relation \eqref{tfassum} may be written as  \begin{equation}
\tf_{m,\kappa} = \hf_{\kappa} + \sum_{\lambda=1}^{\infty} \hf_{\kappa+\lambda b^{m}} \E^{2 \pi \sqrt{-1} \ip{\tvk(\kappa+\lambda b^{m}) \ominus \tvk(\kappa)}{\bsDelta}/b},
\label{tfassumc}
\end{equation}
and the cubature error in \eqref{err1} may be bounded as 
\begin{equation}
\biggabs{ \int_{\cube} f(\bsx) \, \D \bsx - \hI_m(f)} 
= \Biggabs {\sum_{\lambda=1}^{\infty} \hf_{\lambda b^m} \E^{2 \pi \sqrt{-1} \ip{\tvk(\lambda b^m)}{\bsDelta}/b}}
\le \sum_{\lambda=1}^{\infty} \left \lvert \hf_{\lambda b^m}\right \rvert. \label{err2}
\end{equation}
We will use the discrete transform, $\tf_{m,\kappa}$, to estimate true Walsh coefficients, $\hf_{\kappa}$, for $m$ significantly larger than $\lfloor \log_b(\kappa) \rfloor$.

\subsection{Sums of Walsh Series Coefficients and Cone Conditions}
Consider the following sums of the true and approximate Walsh series coefficients.  For $\ell,m \in \N_0$ and $\ell \le m$ let
\begin{gather*}
S_m(f) =  \sum_{\kappa=\left \lfloor b^{m-1} \right \rfloor}^{b^{m}-1} \bigabs{\hf_{\kappa}}, \qquad 
\hS_{\ell,m}(f)  = \sum_{\kappa=\left \lfloor b^{\ell-1} \right \rfloor}^{b^{\ell}-1} \sum_{\lambda=1}^{\infty} \bigabs{ \hf_{\kappa+\lambda b^{m}}}, \\
\wcS_m(f)=\hS_{0,m}(f) + \cdots + \hS_{m,m}(f)=
\sum_{\kappa=b^{m}}^{\infty} \bigabs{\hf_{\kappa}}, \qquad
\tS_{\ell,m}(f) = \sum_{\kappa=\left \lfloor b^{\ell-1}\right \rfloor}^{b^{\ell}-1} \bigabs{\tf_{m,\kappa}}.
\end{gather*}
The first three sums, $S_{m}(f)$, $\hS_{\ell,m}(f)$, and $\wcS_m(f)$, cannot be observed because they involve the true series coefficients. But, the last sum, $\tS_{\ell,m}(f)$, is defined in terms of the discrete Walsh transform and can easily be computed in terms of function values.  The details are described in the Appendix.

We now make critical assumptions about how certain sums provide upper bounds on others.  Let $\ell_* \in \N$ be some fixed integer and $\homega$ and $\wcomega$ be some non-negative valued functions with $\lim_{m \to \infty} \wcomega(m) = 0$ such that $\homega(r)\wcomega(r)<1$ for some $r\in\N$.  Define the cone of integrands
\begin{multline} \label{conecond}
\cc:=\{f \in L^2(\cube) : \hS_{\ell,m}(f) \le \homega(m-\ell) \wcS_m(f),\ \ \ell \le m, \\
\wcS_m(f) \le \wcomega(m-\ell) S_{\ell}(f),\ \  \ell_* \le \ell \le m\}.
\end{multline}
This is a cone because $f \in \cc \implies af \in \cc$ for all real $a$.

The first inequality asserts that the sum of the larger indexed Walsh coefficients bounds a partial sum of the same coefficients.  For example, this means that $\hS_{0,12}$, the sum of the values of the large black dots in Figure \ref{Walshcoeffig}, is no greater than some factor times $\wcS_{12}(f)$, the sum of the values of the gray $\boldsymbol{\times}$. Possible choices of $\homega$ are $\homega(m)=1$ or $\homega(m)=C b^{-\alpha m}$ for some $C>1$ and $0\le \alpha \le 1$. The second inequality asserts that the sum of the smaller indexed coefficients provides an upper bound on the sum of the larger indexed coefficients.  In other words, the fine scale components of the integrand are not unduly large compared to the gross scale components.  In Figure \ref{Walshcoeffig} this means that $\wcS_{12}(f)$ is no greater than some factor times $S_8(f)$, the sum of the values of the black squares.  This implies that $\bigabs{\hf_{\kappa}}$ does not dip down and then bounce back up too dramatically as $\kappa \to \infty$. The reason for enforcing the second inequality only for $\ell \ge \ell_*$ is that for small $\ell$, one might have a coincidentally small $S_\ell(f)$, while $\wcS_m(f)$ is large.

\begin{figure}
\centering
\includegraphics[width=9cm]{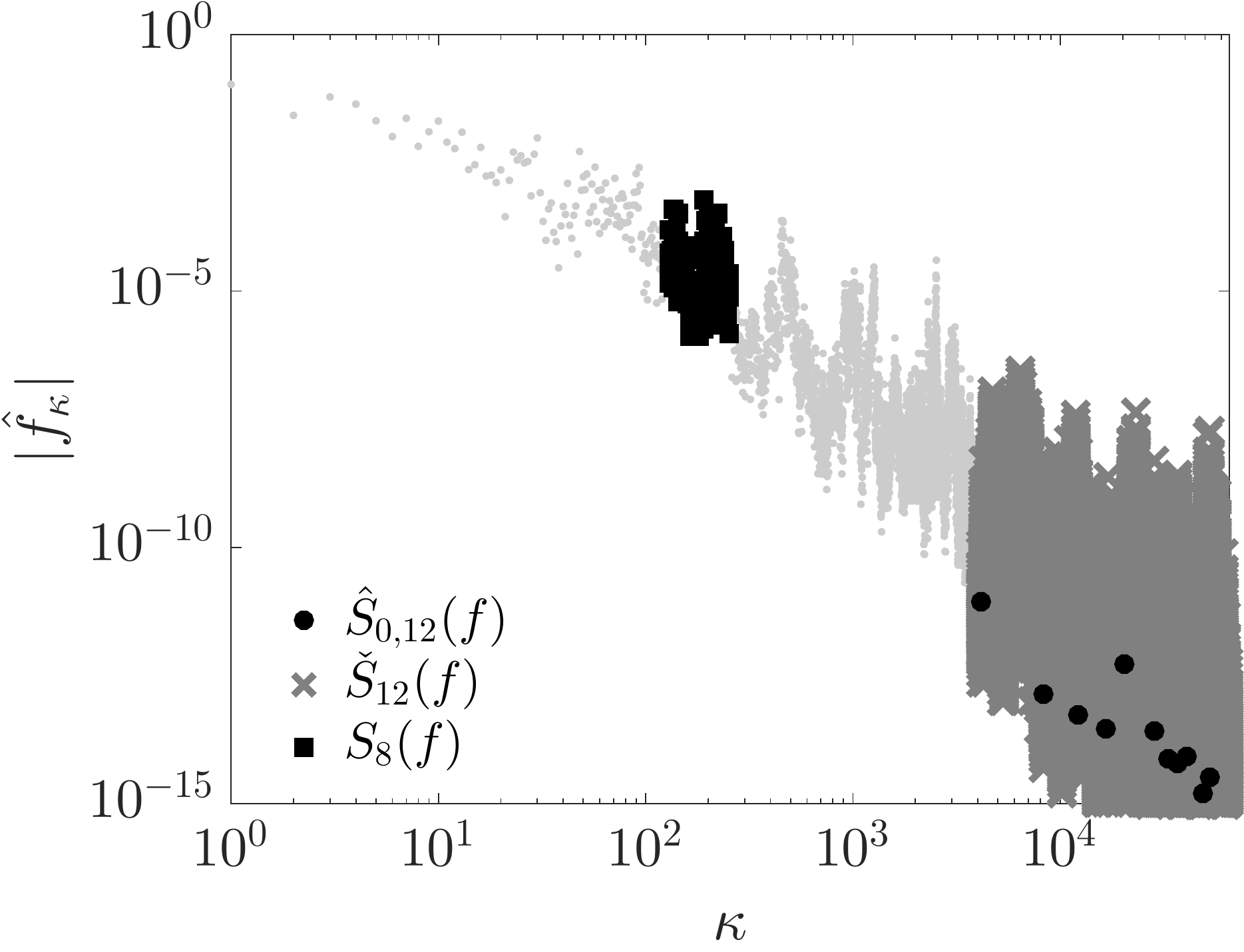}
\caption{The magnitudes of true Walsh coefficients for $f(x)=\mathrm{e}^{-3x}\sin\left(10{x^2}\right)$. \label{Walshcoeffig}}
\end{figure}

The cubature error bound in \eqref{err2} can be bounded in terms of $S_l(f)$, a certain finite sum of the Walsh coefficients for integrands $f$ in the cone $\cc$.  For  $\ell, m \in \N$, $\ell_* \le \ell \le m$, it follows that 
\begin{align}
\nonumber
\biggabs{\int_{\cube} f(\bsx) \, \D \bsx - \hI_m(f) }
&\le \sum_{\lambda=1}^{\infty} \bigabs{\hf_{\lambda b^{m}}} 
= \hS_{0,m}(f) \qquad \text{by \eqref{err2}} \\
& \le \homega(m) \wcS_m(f) \le \homega(m) \wcomega(m-\ell) S_\ell(f). \label{SSbd2}
\end{align}
Thus, the faster $S_\ell(f)$ decays as $\ell \to \infty$, the faster the cubature error must decay.

Unfortunately, the true Walsh coefficients are unknown.  Thus we must bound $S_\ell(f)$ in terms of the observable sum of the approximate coefficients, $\tS_{\ell,m}(f)$.  This is done as follows:
\begin{align}
\nonumber
S_\ell(f) &= \sum_{\kappa=\left \lfloor b^{\ell-1} \right \rfloor}^{b^{\ell}-1} \bigabs{ \hf_{\kappa}} \\
\nonumber
&= \sum_{\kappa=\left \lfloor b^{\ell-1} \right \rfloor}^{b^{\ell}-1} \biggabs{\tf_{m,\kappa} - \sum_{\lambda=1}^{\infty} \hf_{\kappa+\lambda b^{m}} \E^{2 \pi \sqrt{-1} \ip{\tvk(\kappa+\lambda b^{m}) \ominus \tvk(\kappa)}{\bsDelta}/b}} \qquad \text{by \eqref{tfassumc}}\\
\nonumber
&\le \sum_{\kappa=\left \lfloor b^{\ell-1} \right \rfloor}^{b^{\ell}-1} \bigabs{\tf_{m,\kappa}} + \sum_{\kappa=\left \lfloor b^{\ell-1} \right \rfloor}^{b^{\ell}-1} \sum_{\lambda=1}^{\infty} \bigabs{\hf_{\kappa+\lambda b^{m}}} 
= \tS_{\ell,m}(f) + \hS_{\ell,m}(f) \\
\nonumber
&\le \tS_{\ell,m}(f) + \homega(m-\ell) \wcomega(m-\ell) S_\ell(f) \qquad \text{by \eqref{conecond}}, \\
S_{\ell}(f) & \le \frac{\tS_{\ell,m}(f)}{1 - \homega(m-\ell) \wcomega(m-\ell)} \qquad \text{provided that } \homega(m-\ell) \wcomega(m-\ell) < 1. \label{SSbd}
\end{align}
Combining \eqref{SSbd2} with \eqref{SSbd} leads to the following conservative upper bound on the cubature error for $\ell, m \in \N$, $\ell_* \le \ell \le m$: 
\begin{equation}
\biggabs{\int_{\cube} f(\bsx) \, \D \bsx - \hI_m(f) }
\le \frac{\tS_{\ell,m}(f)\homega(m) \wcomega(m-\ell)}{1 - \homega(m-\ell) \wcomega(m-\ell)}. \label{errbd}
\end{equation}
This error bound suggests the following algorithm. 

\subsection{An Adaptive Cubature Algorithm and Its Cost}

\begin{algo}[Adaptive Digital Sequence Cubature, \texttt{cubSobol\_g}] \label{adapalgo} Given the parameter $\ell_* \in \N$ and the functions $\homega$ and  $\wcomega$ that define the cone $\cc$ in \eqref{conecond}, choose the parameter $r \in \N$ such that $\homega(r)\wcomega(r)<1$.  Let $\fC(m): = \homega(m)\wcomega(r)/[1 - \homega(r)\wcomega(r)]$ and $m=\ell_*+r$.
Given a tolerance, $\varepsilon$, and a routine that produces values of the integrand, $f$, do the following:

\begin{description}
\item[\textbf{Step 1.}] Compute the sum of the discrete Walsh coefficients, $\tS_{m-r,m}(f)$ according to Algorithm \ref{pointeralgo}.  

\item[\textbf{Step 2.}] Check whether the error tolerance is met, i.e., whether $\fC(m)  \tS_{m-r,m}(f) \le \varepsilon$. If so, then return the cubature $\hI_m(f)$ defined in \eqref{cubaturedef} as the answer.

\item[\textbf{Step 3.}] Otherwise, increment $m$ by one, and go to Step 1.
\end{description}

\end{algo}

There is a balance to be struck in the choice of $r$.  Choosing $r$ too large causes the error bound to depend on the Walsh coefficients with smaller indices, which may be large, even thought the Walsh coefficients determining the error are small.  Choosing $r$ too large makes $\homega(r)\wcomega(r)$ large, and thus the inflation factor, $\fC$, large to guard against aliasing.

\begin{theorem} \label{adapalgothm} If the integrand, $f$, lies in the cone, $\cc$, then the Algorithm \ref{adapalgo} is successful: 
\[
\biggabs{\int_{\cube} f(\bsx) \D \bsx - \hI_m(f)} \le \varepsilon.
\]
The number of integrand values required to obtain this answer is $b^m$, where the following upper bound on $m$ depends on the tolerance and unknown decay rate of the Walsh coefficients.
\[
m \le \min \{m' \ge \ell_*+r : \fC(m') [1+ \homega(r) \wcomega(r)] S_{m'-r}(f) \le \varepsilon \}
\]
The computational cost of this algorithm beyond that of obtaining the integrand values is $\Order(m b^m )$ to compute the discrete Walsh transform.
\end{theorem}

\begin{proof}
The success of this algorithm comes from applying \eqref{errbd}.  To bound the number of integrand values required note that argument leading to \eqref{SSbd} can be modified to provide an upper bound on $\tS_{\ell,m}(f)$ in terms of $S_{\ell}(f)$:
\begin{align}
\nonumber
\tS_{\ell,m}(f) &= \sum_{\kappa=\left \lfloor b^{\ell-1} \right \rfloor}^{b^{\ell}-1} \bigabs{ \tf_{m,\kappa}}\\
\nonumber
& = \sum_{\kappa=\left \lfloor b^{\ell-1} \right \rfloor}^{b^{\ell}-1} \biggabs{\hf_{\kappa} + \sum_{\lambda=1}^{\infty} \hf_{\kappa+\lambda b^{m}} \E^{2 \pi \sqrt{-1} \ip{\tvk(\kappa+\lambda b^{m}) \ominus \tvk(\kappa)}{\bsDelta}/b}} \qquad \text{by \eqref{tfassumc}}\\
\nonumber
&\le\sum_{\kappa=\left \lfloor b^{\ell-1} \right \rfloor}^{b^{\ell}-1} \bigabs{\hf_{\kappa}} + \sum_{\kappa=\left \lfloor b^{\ell-1} \right \rfloor}^{b^{\ell}-1} \sum_{\lambda=1}^{\infty} \bigabs{\hf_{\kappa+\lambda b^{m}}} 
= S_{\ell}(f) + \hS_{\ell,m}(f) \\
\nonumber
&\le [1  + \homega(m-\ell) \wcomega(m-\ell)] S_\ell(f) \qquad \text{by \eqref{conecond}}. \label{SSbd3}
\end{align}
Thus, the upper bound on the error in Step 2 of Algorithm \ref{adapalgo}, is itself bounded above by $\fC(m) [1  + \homega(r) \wcomega(r)] S_{m-r}(f)$.  Therefore, the stopping criterion in Step 2 must be satisfied no later than when this quantity falls below $\varepsilon$. 

The computation of the discrete Walsh transform and $\tS_{m-r,m}(f)$ is described in Algorithm \ref{pointeralgo} in the Appendix.  The cost of this algorithm is $\Order(mb^m)$ operations. \hfill \qed
\end{proof}

\section{Numerical Experiments} \label{numexpsec}

Algorithm \ref{adapalgo} has been implemented in MATLAB code as the function \texttt{cubSobol\_g}. It is included in our Guaranteed Automatic Integration Library (GAIL) \cite{ChoEtal15a}.  Our \texttt{cubSobol\_g} utilizes MATLAB's built-in Sobol' sequences, so $b=2$.  The default algorithm parameters are 
\[
\ell_*=6, \qquad r=4, \qquad \fC(m)=5 \times 2^{-m},
\]
and mapping $\tvk$ is fixed heuristically according to Algorithm \ref{pointeralgo}. Fixing $\fC$ partially determines $\homega$ and $\wcomega$ since $\homega(m) = \fC(m)/\homega(r)$ and $\homega(r)\wcomega(r) = \fC(r)/[1+\fC(r)]$.

We have tried \texttt{cubSobol\_g} on an example from \cite{Kei96}:
\begin{equation} \label{KeisterEx}
I=\int_{\R^d} \E^{-\norm{\bst}^2} \cos(\norm{\bst}) \, \D \bst 
= \pi^{d/2}\int_{\cube} \cos\left(\sqrt{\frac 12 \sum_{j=1}^d [\Phi^{-1}(x_j)]^2} \right) \, \D \bsx,
\end{equation}
where $\Phi$ is the standard Gaussian distribution function.  We generated $1000$ IID random values of the dimension $d=\lfloor \E^D \rfloor$ with $D$ being uniformly distributed between $0$ and $\log(20)$.  Each time \texttt{cubSobol\_g} was run, a different scrambled and shifted Sobol' sequence was used.  The tolerance was met about $97\%$ of the time and failures were more likely among the higher dimensions. For those cases where the tolerance was not met, mostly the larger dimensions, the integrand lay outside the cone $\cc$.  Our choice of $\tilde{\boldsymbol{k}}$ via Algorithm \ref{pointeralgo} depends somewhat on the particular scrambling and digital shift, so the definition of $\cc$ also depends mildly on these.

\begin{figure}
\centering
\includegraphics[width=7cm]{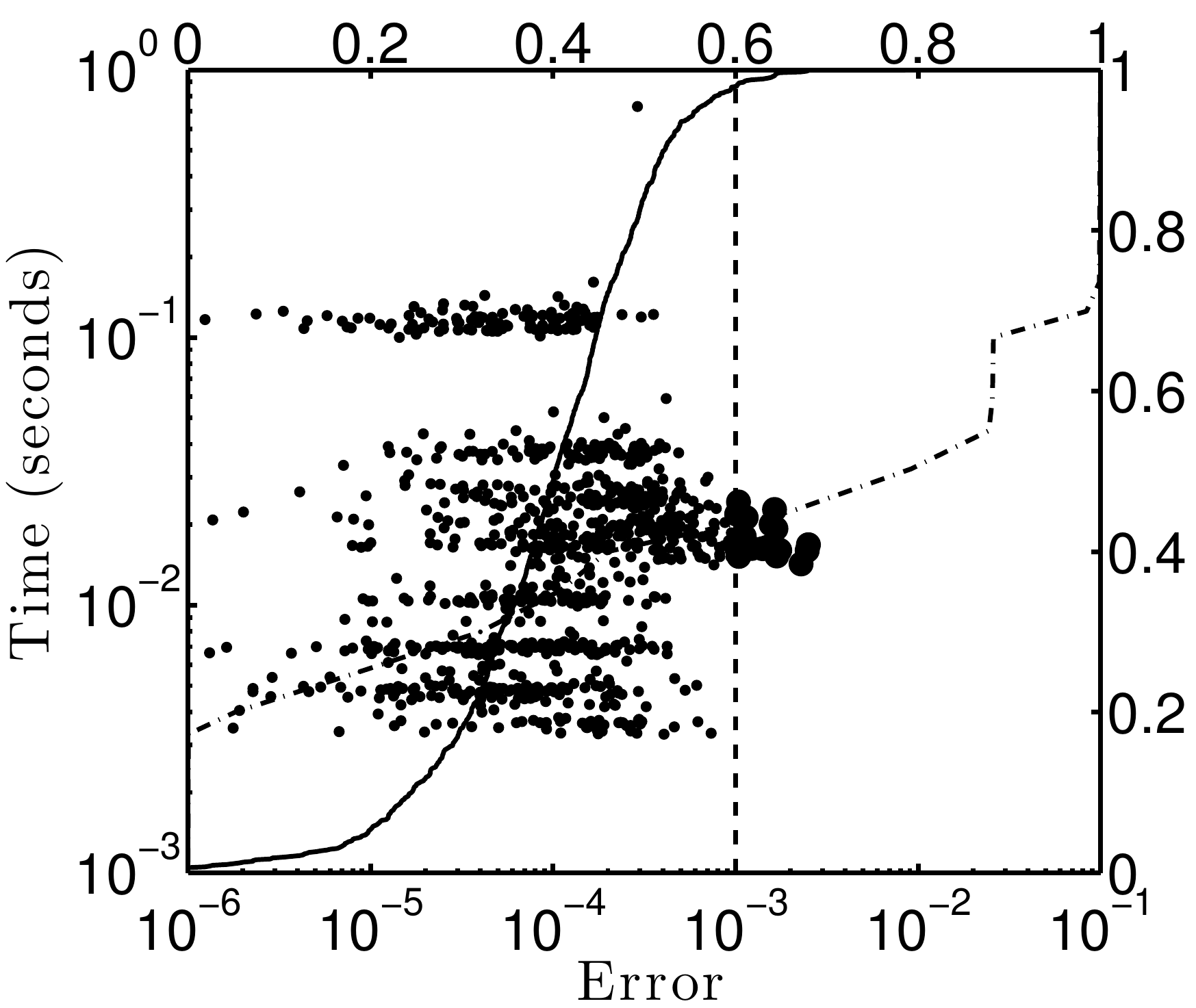}
\caption{Time required and error observed for \texttt{cubSobol\_g} (Algorithm \ref{adapalgo}) for the Keister example,  \eqref{KeisterEx}. Small dots denote the time and error when the tolerance of $\varepsilon=0.001$ was met.  Large dots denote the time and error when the tolerance was not met.  The solid line denotes the empirical distribution function of the error, and the dot-dashed line denotes the empirical distribution function of the time. \label{numexamplesfig}} 
\end{figure}

\section{Discussion}

There are few quasi-Monte Carlo cubature algorithms available that adaptively determine the sample size needed based on integrand values.  The chief reason is that reliable error estimation for quasi-Monte Carlo is difficult.  Quasi-standard error has serious drawbacks, as explained in \cite{Owe06a}.  Internal replications have no explicit theory.  IID replications of randomized quasi-Monte Carlo rules are sometimes used, but one does not know how many replications are needed.

The proposed error bound and adaptive algorithm here are practical and have theoretical justification. The conditions imposed on the sums of the (true) Fourier-Walsh coefficients make it possible to bound the cubature error in terms of discrete Fourier-Walsh coefficients.  The set of integrands satisfying these conditions is a non-convex cone \eqref{conecond}, thereby placing us in a setting where adaption has the opportunity to be beneficial.

Problems requiring further consideration include how to choose the default parameters for Algorithm \ref{adapalgo}.  We would also like to extend our algorithm and theory to the case of relative error.

\begin{acknowledgement}
This work was partially supported by US National Science Foundation grants DMS-1115392, DMS-1357690, and DMS-1522687.  The authors thank Ronald Cools and Dirk Nuyens for organizing MCQMC 2014.  We thank Sergei Kucherenko and Art Owen for organizing the special session in honor of Ilya M. Sobol'.  We are grateful for Professor Sobol's many contributions to MCQMC and related fields.   The suggestions made by Sou-Cheng Choi, Yuhan Ding, Lan Jiang, and the anonymous referees to improve this manuscript are greatly appreciated.  
\end{acknowledgement}


\begin{thebibliography}{10}
\providecommand{\url}[1]{{#1}}
\providecommand{\urlprefix}{URL }
\expandafter\ifx\csname urlstyle\endcsname\relax
  \providecommand{\doi}[1]{DOI~\discretionary{}{}{}#1}\else
  \providecommand{\doi}{DOI~\discretionary{}{}{}\begingroup
  \urlstyle{rm}\Url}\fi

\bibitem{Caf98}
Caflisch, R.E.: {M}onte {C}arlo and quasi-{M}onte {C}arlo methods.
\newblock Acta Numerica \textbf{7}, 1--49 (1998).

\bibitem{ChoEtal15a}
Choi, S.C.T., Ding, Y., Hickernell, F.J., Jiang, L., {Jim\'enez Rugama},
  {\relax Ll}.A., Tong, X., Zhang, Y., Zhou, X.: {GAIL}: {G}uaranteed
  {A}utomatic {I}ntegration {L}ibrary (versions 1.0--2.1).
\newblock MATLAB software ({2013--2015}).
\newblock \urlprefix\url{https://github.com/GailGithub/GAIL_Dev}.

\bibitem{DicPil10a}
Dick, J., Pillichshammer, F.: Digital Nets and Sequences: Discrepancy Theory
  and Quasi-{M}onte {C}arlo Integration.
\newblock Cambridge University Press, Cambridge (2010).

\bibitem{Hic97a}
Hickernell, F.J.: A generalized discrepancy and quadrature error bound.
\newblock Mathematics of Computation \textbf{67}, 299--322 (1998).

\bibitem{HicSloWas03c}
Hickernell, F.J., Sloan, I.H., Wasilkowski, G.W.: On strong tractability of
  weighted multivariate integration.
\newblock Mathematics of Computation \textbf{73}, 1903--1911 (2004).

\bibitem{HicSloWas03b}
Hickernell, F.J., Sloan, I.H., Wasilkowski, G.W.: On tractability of
  weighted integration for certain {B}anach spaces of functions.
\newblock In: Niederreiter  \cite{Nie04}, pp. 51--71.

\bibitem{HicSloWas03a}
Hickernell, F.J., Sloan, I.H., Wasilkowski, G.W.: On tractability of
  weighted integration over bounded and unbounded regions in {$\mathbb{R}^s$}.
\newblock Mathematics of Computation \textbf{73}, 1885--1901 (2004).

\bibitem{HicSloWas03e}
Hickernell, F.J., Sloan, I.H., Wasilkowski, G.W.: The strong tractability of
  multivariate integration using lattice rules.
\newblock In: Niederreiter  \cite{Nie04}, pp. 259--273.

\bibitem{JimHic16a}
{Jim\'enez Rugama}, {\relax Ll}.A., Hickernell, F.J.: Adaptive
  multidimensional integration based on rank-1 lattices.
\newblock In: R.~Cools, D.~Nuyens (eds.) {M}onte {C}arlo and Quasi-{M}onte
  {C}arlo Methods 2014. Springer-Verlag, Berlin (2015+).
\newblock To appear, arXiv:1411.1966.

\bibitem{Kei96}
Keister, B.D.: Multidimensional quadrature algorithms.
\newblock Computers in Physics \textbf{10}, 119--122 (1996).

\bibitem{Lem09a}
Lemieux, C.: {M}onte {C}arlo and quasi-{M}onte {C}arlo Sampling.
\newblock Springer Science+Business Media, Inc., New York (2009).

\bibitem{Nie92}
Niederreiter, H.: Random Number Generation and Quasi-{M}onte {C}arlo Methods.
\newblock CBMS-NSF Regional Conference Series in Applied Mathematics. SIAM,
  Philadelphia (1992).

\bibitem{Nie04}
Niederreiter, H. (ed.): {M}onte {C}arlo and Quasi-{M}onte {C}arlo Methods 2002.
  Springer-Verlag, Berlin (2004).

\bibitem{NovWoz10a}
Novak, E., Wo{\'{z}}niakowski, H.: Tractability of Multivariate Problems
  {V}olume II: {S}tandard Information for Functionals.
\newblock No.~12 in EMS Tracts in Mathematics. European Mathematical Society,
  Z\"urich (2010).

\bibitem{Owe06a}
Owen, A.B.: On the {W}arnock-{H}alton quasi-standard error.
\newblock Monte Carlo Methods and Applications \textbf{12}, 47--54 (2006).

\end{thebibliography}

\section*{Appendix:  Fast Computation of the Discrete Walsh Transform}

Let $y_0, y_1, \ldots$ be some data.  Define $Y_\nu^{(m)}$ for $\nu=0, \ldots, b^{m}-1$ as follows:
\[
Y^{(m)}_{\nu}  := \frac{1}{b^m} \sum_{i=0}^{b^m-1} \E^{-2 \pi \sqrt{-1} \sum_{\ell=0}^{m-1} \nu_\ell i_\ell/b} y_i
= \frac{1}{b^m} \sum_{i_{m-1}=0}^{b-1} \cdots \sum_{i_0=0}^{b-1} \E^{-2 \pi \sqrt{-1} \sum_{\ell=0}^{m-1} \nu_\ell i_\ell/b} y_i,
\]
where $i=i_0 + i_1 b + \cdots i_{m-1} b^{m-1}$ and $\nu=\nu_0 + \nu_1 b + \cdots \nu_{m-1} b^{m-1}$.  For all $i_j, \nu_j \in \F_b$,  $j,\ell=0, \ldots, m-1$,  recursively define
\[
Y_{m,0}(i_0, \ldots, i_{m-1}):= y_i,
\]
\begin{multline*}
Y_{m,\ell+1}(\nu_0, \ldots, \nu_{\ell},i_{\ell+1}, \ldots, i_{m-1}) \\
:= \frac{1}{b} \sum_{i_\ell=0}^{b-1} \E^{-2 \pi \sqrt{-1} \nu_\ell i_\ell/b } Y^{(m)}_{m,\ell}(\nu_1, \ldots, \nu_{\ell-1},i_{\ell}, \ldots, i_{m-1}).
\end{multline*}
This allows us to identify $Y^{(m)}_{\nu}=Y_{m,m}(\nu_0, \ldots, \nu_{m-1})$.  By this iterative process one can compute $Y^{(m)}_0, \ldots, Y^{(m)}_{b^m-1}$ in only $\Order(m b^m)$ operations.

Note also, that $Y_{m+1,m}(\nu_0, \ldots, \nu_{m-1},0)=Y_{m,m}(\nu_0, \ldots, \nu_{m-1})=Y^{(m)}_{\nu}$.  This means that the work done to compute $Y^{(m)}_{\nu}$ can be used to compute $Y^{(m+1)}_{\nu}$.

Next, we relate the $Y_{\nu}$ to the discrete Walsh transform of the integrand $f$. For every $\bsk \in \N_0^d$ and every digital sequence $\cp_{\infty} = \{\bsz_i\}_{i=0}^{\infty}$, let 
\begin{equation} \label{numapdefeq}
\tnu_0(\bsk) := 0, \qquad \tnu_m(\bsk) := \sum_{\ell=0}^{m-1} \ip{\bsk}{\bsz_{b^\ell}} b^{\ell} \in \natm, \quad m \in \N.
\end{equation}
If we set $y_i=f(\bsz_i+\bsDelta)$, and if $\tnu_m(\bsk)=\nu$, then 
\begin{align}
\nonumber
\tf_m(\bsk) 
&= \frac{1}{b^m} \sum_{i=0}^{b^m-1} \E^{-2 \pi \sqrt{-1} \ip{\bsk}{\bsz_i\oplus \bsDelta}/b} y_i \\
\nonumber
&= \frac{\E^{-2 \pi \sqrt{-1} \ip{\bsk}{\bsDelta}/b}}{b^m} \sum_{i=0}^{b^m-1} \E^{-2 \pi \sqrt{-1} \ip{\bsk}{\bsz_i}/b} y_i \qquad \text{by \eqref{bilinearlinxprop}} \\
\nonumber
&= \frac{\E^{-2 \pi \sqrt{-1} \ip{\bsk}{\bsDelta}/b}}{b^m} \sum_{i=0}^{b^m-1} \E^{-2 \pi \sqrt{-1} \ip{\bsk}{\sum_{j=0}^{m-1} i_j\bsz_{b^j}}/b} y_i \qquad \text{by \eqref{cpinfvector}}  \\ \nonumber
&= \frac{\E^{-2 \pi \sqrt{-1} \ip{\bsk}{\bsDelta}/b}}{b^m} \sum_{i=0}^{b^m-1} \E^{-2 \pi \sqrt{-1} \sum_{j=0}^{m-1} i_j \ip{\bsk}{\bsz_{b^j}}/b} y_i \qquad \text{by \eqref{bilinearlinxprop}}  \\ 
\nonumber 
& = \frac{\E^{-2 \pi \sqrt{-1} \ip{\bsk}{\bsDelta}/b}}{b^m} \sum_{i=0}^{b^m-1} \E^{-2 \pi \sqrt{-1} \sum_{\ell=0}^{m-1} \nu_\ell i_\ell/b} y_i \qquad \text{by \eqref{numapdefeq}}\\
& = \E^{-2 \pi \sqrt{-1} \ip{\bsk}{\bsDelta}/b} Y^{(m)}_{\nu}.
\end{align}

Using the notation in Sect.\ \ref{ErrEstsec}, for all $m \in \N_0$ define a pointer $\rnu_m:\natm \to \natm$ as $\rnu_m(\kappa) := \tnu_m(\tvk(\kappa))$. It follows that 
\begin{gather}
\nonumber
\tf_{m,\kappa} = \tf_m(\tvk(\kappa)) = \E^{-2 \pi \sqrt{-1} \ip{\bsk}{\bsDelta}/b} Y^{(m)}_{\rnu_m(\kappa)}, \\
\label{tSrnu}
\tS_{\ell,m}(f) = \sum_{\kappa=b^{\ell-1}}^{b^{\ell}-1} \bigabs{ \tf_{m,\kappa}} = \sum_{\kappa=b^{\ell-1}}^{b^{\ell}-1} \Bigabs{ Y^{(m)}_{\rnu_m(\kappa)}}.
\end{gather}
The quantity $\tS_{m-r,m}(f)$ is the key to the stopping criterion in Algorithm \ref{adapalgo}.

If the map $\tvk:\N_0 \to \N_0^d$ defined in Algorithm \ref{wavenummapalgo} is known explicitly, then specifying $\rnu_m$ is straightforward.  However, in practice the bookkeeping involved in constructing $\tvk$ might be tedious, so we take a data-dependent approach to constructing the pointer $\rnu_m(\kappa)$ for $\kappa \in \natm$ directly, which then defines $\tvk$ implicitly. 

\begin{algo} \label{pointeralgo} Let $r \in \N$ be fixed.  Given the input $m \in \N_0$, the discrete Walsh coefficients $Y^{(m)}_\nu$ for $\nu \in \natm$, and also the pointer $\rnu_{m-1}(\kappa)$ defined for $\kappa \in \N_{0,m-1}$, provided $m>0$, 
\begin{description}  
\item[\textbf{Step 1.}]  If $m=0$, then define $\rnu(0)=0$ and go to Step 4.  
\item[\textbf{Step 2.}]  Otherwise, if $m \ge 1$, then initialize $\rnu_m(\kappa)=\rnu_{m-1}(\kappa)$ for $\kappa \in \N_{0,m-1}$ and $\rnu_m(\kappa)=\kappa$ for $\kappa =b^{m-1}, \ldots, b^{m}-1$.

\item[\textbf{Step 3.}]  For $\ell =m-1, m-2, \ldots, \max(1,m-r)$,\\
\hspace*{1.5cm} for $\kappa = 1, \ldots, b^{\ell} -1 $ \\
\hspace*{2cm} Find $a^*$ such that $\Bigabs{Y^{(m)}_{\rnu_m(\kappa+a^* b^{\ell})}} \ge \Bigabs{Y^{(m)}_{\rnu_m(\kappa+a b^{\ell})}} $ for all $a \in \F_b$.\\
\hspace*{2cm} Swap the values of $\rnu_m(\kappa)$ and $\rnu_m(\kappa+a^* b^{\ell})$.

\item[\textbf{Step 4.}]  Return $\rnu_m(\kappa)$ for $\kappa \in \natm$.  If $m \ge r$, then compute $\tS_{m-r,r}(f)$ according to \eqref{tSrnu}, and return this value as well.

\end{description}
\end{algo}

\begin{lemma}  Let $\cp_{m,\kappa}^{\perp}:=\{\bsk \in \N_0^d : \tnu_m(\bsk)=\rnu_m(\kappa)\}$ for $\kappa \in \natm$, $m \in \N_0$, where $\rnu_m$ is given by  Algorithm \ref{pointeralgo}.  Then we implicitly have defined the map $\tvk$ in the sense that any map $\tvk: \N_{0,m} \to \N_0^d$ that chooses $\tvk(0) = \bszero \in \cp_{m,0}^{\perp}$, and  $\tvk(\kappa) \in \cp_{m,\kappa}$ for all $\kappa=1, \ldots, b^{m}-1$ gives the same value of $S_{m-r,r}(f)$.  It is also consistent with Algorithm \ref{wavenummapalgo} for $\kappa \in \N_{0,m-r}$.
\end{lemma}
\begin{proof}  The constraint that $\tvk(\kappa) \in \cp_{m,\kappa}$ implies that $S_{m-r,r}(f)$ is invariant under all $\tvk$ chosen according to the assumption that $\tvk(\kappa) \in \cp_{m,\kappa}$. By definition $\bszero \in \cp_{m,0}^{\perp}$ remains true for all $m$ for Algorithm \ref{pointeralgo}.

The remainder of the proof is to show that choosing $\tvk(\kappa)$ by the hypothesis of this lemma is consistent with Algorithm \ref{wavenummapalgo}. To do this we show that for $m \in N_0$
\begin{equation}
\bsk \in \cp_{m,\kappa}^\perp,\ \bsl \in \cp_{m,\kappa + a b^{\ell}}^\perp \implies \bsk \ominus \bsl \in \cp_{\ell}^{\perp} \qquad \text{for all } \kappa=1, \ldots, b^{\ell}, \ell < m,
\label{pointpfconda}
\end{equation}
and that 
\begin{equation}
\cp_{m,\kappa}^\perp \supset \cp_{m+1,\kappa}^\perp \supset \cdots \qquad \text{for } \kappa \in \N_{0,m-r} \text{ provided }m \ge r. \label{pointpfcondb}
\end{equation}
The proof proceeds by induction.  Since $\cp_{0,0}^{\perp}=\N_0^d$, the above two conditions are satisfied automatically.

If they are satisfied for $m-1$ (instead of $m$), then the initialization stage in Step 2 of Algorithm \ref{pointeralgo} preserves \eqref{pointpfconda} for $m$.  The swapping of $\rnu_m(\kappa)$ and $\rnu_m(\kappa+a^* b^{\ell})$ values in Step 3 also preserves \eqref{pointpfconda}.  Step 3 may cause $\cp_{m-1,\kappa}^\perp \cap \cp_{m,\kappa}^\perp = \emptyset$ for some larger values of $\kappa$, but the constraint on the values of $\ell$ in Step 3 mean that \eqref{pointpfcondb} is preserved.
\qed
\end{proof}

\end{document}